\documentclass[12pt,reqno]{amsart}
\usepackage{amsmath}
\usepackage{amssymb}
\usepackage{amstext}
\usepackage{a4wide}
\usepackage{graphicx}
\usepackage{mathrsfs}
\allowdisplaybreaks \numberwithin{equation}{section}
\usepackage{color}
\usepackage{cases}

\numberwithin{equation}{section}

\newtheorem{theorem}{Theorem}[section]

\newtheorem{corollary}[theorem]{Corollary}
\newtheorem{lemma}[theorem]{Lemma}

\theoremstyle{definition}

\newtheorem{definition}[theorem]{Definition}

\theoremstyle{remark}
\newtheorem{remark}[theorem]{Remark}

\begin{document}

\title
{Asymptotic Behavior of Global vortex rings}

 \author{Daomin Cao, Jie Wan, Guodong Wang,  Weicheng Zhan}

\address{Institute of Applied Mathematics, Chinese Academy of Sciences, Beijing 100190, and University of Chinese Academy of Sciences, Beijing 100049,  P.R. China}
\email{dmcao@amt.ac.cn}

\address{Institute of Applied Mathematics, Chinese Academy of Sciences, Beijing 100190, and University of Chinese Academy of Sciences, Beijing 100049,  P.R. China}
\email{wanjie15@mails.ucas.edu.cn}

\address{Institute for Advanced Study in Mathematics, Harbin Institute of Technology, Harbin 150001, P.R. China}
\email{wangguodong14@mails.ucas.edu.cn}

\address{Institute of Applied Mathematics, Chinese Academy of Sciences, Beijing 100190, and University of Chinese Academy of Sciences, Beijing 100049,  P.R. China}
\email{zhanweicheng16@mails.ucas.ac.cn}


\begin{abstract}
 In this paper, we are concerned with nonlinear desingularization of steady
vortex rings in $\mathbb{R}^3$ with given general nonlinearity $f$. Using the improved vorticity method,
we construct a family of steady vortex rings which constitute a desingularization of the
classical circular vortex filament in the whole space. The requirements on $f$ are very general, and $f$ may have a simple discontinuity at the origin. Some qualitative
and asymptotic properties are also established.
\end{abstract}

\maketitle

\section{Introduction}

The motion of an incompressible steady Euler fluid in $\mathbb{R}^3$ is governed by the following Euler equations

\begin{equation}\label{1-1}
~(\mathbf{v}\cdot\nabla)\mathbf{v}=-\nabla P,
\end{equation}
\begin{equation}\label{1-2}
 \nabla\cdot\mathbf{v}=0,
\end{equation}
where $\mathbf{v}=(v_1,v_2,v_3)$ is the velocity field and $P$ is the scalar pressure. Let $\{\mathbf{e}_r, \mathbf{e}_\theta, \mathbf{e}_z\}$ be the usual cylindrical coordinate frame. Then if the velocity field $\mathbf{v}$ is axisymmetric, i.e., $\mathbf{v}$ does not depend on the $\theta$ coordinate, it can be expressed in the following way
\begin{equation*}
  \mathbf{v}=v^r(r,z)\mathbf{e}_r+v^\theta(r,z)\mathbf{e}_\theta+v^z(r,z)\mathbf{e}_z.
\end{equation*}
 The component $v^\theta$ in the $\mathbf{e}_\theta$ direction is called the swirl velocity. Let $\pmb{\omega}:=\nabla\times\mathbf{v}$ be the corresponding vorticity field. We shall refer to an axisymmetric non-swirling flow~($v^\theta \equiv 0$ ) as ``vortex ring'' if there is a toroidal region inside of which $\boldsymbol{\omega}\not = 0$, and outside of which $\boldsymbol{\omega}= 0$. Note that the conservation of mass equation \eqref{1-2} furnishes a Stokes stream function $\Psi$ such that
 \begin{equation}\label{vvv}
     \mathbf{v}=\frac{1}{r}\left(-\frac{\partial\Psi}{\partial z}\mathbf{e}_r+\frac{\partial\Psi}{\partial r}\mathbf{e}_z\right).
 \end{equation}
  In terms of the Stokes stream function $\Psi$, the problem can be reduced to a free boundary problem on the half plane $\Pi=\{(r,z)\mid r>0\}$ of the form:
\begin{numcases}
{(\mathscr{P})\ \ \  }
\label{1-5} \mathcal{L}\Psi =0 \,\ \, \ \ \ \ \ \ \  \ \, &\text{in}\  $\Pi \backslash A$,  \\
\label{1-6} \mathcal{L}\Psi=\lambda f(\Psi) \ \ \ \  &\text{in}\  $A$,\\
\label{1-7}  \Psi(0,z)=-\mu \le 0, \ \ \ \ \ \ \ \ \ \ \ \ \ \ \, &\\
\label{1-8} \Psi=0 \ \ \ \  &\text{on}\  $\partial A$,\\
\label{1-9} \frac{1}{r}\frac{\partial \Psi}{\partial r} \to -\mathscr{W}\ \text{and}\ \frac{1}{r}\frac{\partial \Psi}{\partial z} \to 0\ \ \text{as}\ \ r^2+z^2\to \infty,&
\end{numcases}
where
\begin{equation*}
  \mathcal{L}:=-\frac{1}{r}\frac{\partial}{\partial r}\Big(\frac{1}{r}\frac{\partial}{\partial r}\Big)-\frac{1}{r^2}\frac{\partial^2}{\partial z^2}.
\end{equation*}
Here the vorticity function $f$ and the vortex-strength parameter $\lambda>0$ are prescribed; $A$ is the (a priori unknown) cross-section of the vortex ring; $\mu$ is called the flux constant measuring the flow rate between the $z$-axis and $\partial A$; The constant $\mathscr{W}>0$ is the propagation speed, and the condition \eqref{1-9} means that the limit of the velocity field $\mathbf{v}$ at infinity is $-\mathscr{W} \mathbf{e}_z$. For a detailed presentation of this model the reader is referred to \cite{BF1}. We shall say that $\Psi \in C^1(\Pi)\cap C^2(\Pi\backslash \partial A)$ is a classical solution of $(\mathscr{P})$ if it solves the first two equations in $(\mathscr{P})$ almost everywhere. We remark that once the Stokes stream function $\Psi$
is obtained, one can easily construct the corresponding solutions of the original
variables $(\mathbf{v}, P)$. In fact, let $F:\mathbb{R}\to \mathbb{R}$ satisfy $F'=f$. Then once we find the Stokes stream function $\Psi$, the velocity of the flow is given by \eqref{vvv}, and the pressure is given by
$P=-F(\Psi)-\frac{1}{2}|\mathbf{v}|^2$, and the corresponding vorticity field is given by $\pmb{\omega}=\lambda rf(\Psi)\mathbf{e}_\theta$.

The study of vortex rings can be traced back to the works of Helmholtz \cite{He} in 1858 and Kelvin \cite{Tho} in 1867. Global existence of vortex rings was first established by Fraenkel and Berger \cite{BF1}. However, the vortex-strength parameter $\lambda$ in \cite{BF1} arises as a Lagrange multiplier and hence is left undetermined. After \cite{BF1}, several works are devoted to further studying the existence of vortex rings. See, e.g., \cite{AS,BB,B1,BF1,FT,Ni,Yang2} and references therein. In addition, \cite{Lim95, MGT, Sa92} are some good historical reviews of vortex rings.

The purpose of this paper is to investigate the asymptotic behaviour of the solution pair $(\Psi^\lambda, A^\lambda)$ of $(\mathscr{P})$ as $\lambda \to +\infty$. More precisely, we shall construct a family of steady vortex rings under some assumptions on $f$, which constitute a desingularization of the classical circular vortex filament in $\mathbb{R}^3$. The class of $f$ we consider here includes all $s_+^p$ with $p>0$. The Stokes stream function $\Psi^\lambda$ will bifurcate from the Green's funtion of $\mathcal{L}$ as $\lambda\to +\infty$. This kind of bifurcation phenomenon is called ``nonlinear desingularization''(refer to \cite{BF2}). For some special nonlinearity $f$, there are already some desingularization results. In \cite{FT}, Friedman and Turkington obtained some desingularization results of steady vortex rings when $f\equiv \text{contant}$. In \cite{Ta}, Tadie studied the asymptotic behaviour by letting the flux constant $\mu$ diverge. In \cite{Yang2}, Yang has studied asymptotic behaviour of vortex rings with some given $f$. We mention that their assumptions on $f$ are stronger than ours. One key assumption is that they need $f$ to satisfy the Ambrosetti-Rabinowitz condition (see $(\text{f}_2)$ in \cite{Yang2}). Our approach does not need this condition. Moreover, their limiting objects are degenerate vortex rings with vanishing circulation.  Our result provides a desingularization of singular vortex filaments with nonvanishing vorticity (see Section 2 below). In \cite{DV}, de Valeriola and Van Schaftingen  also studied desingularization of steady vortex rings with $f(s)=s_+^p$ for $p>1$. In such a case, our results are similar to theirs\,(cf.\,Theorem 1 in \cite{DV}). Our work here may be regarded as an extension of \cite{DV}. Recently, Cao et al. \cite{CWZ} further investigated desingularization of vortex rings when $f\equiv \text{contant}$ and generalized the results in \cite{FT} to some extent. Besides the results mentioned above, we do not know any desingularization results for steady vortex rings in the whole space. Similar problems confined in other domains can be found in \cite{CWZ,D1, D2, DV, Ta}.

The paper is organized as follows. In Section 2, we state our main result. The proof of Theorem \ref{thm} is given in Section 3.

\section{Main results}

 Throughout the sequel we shall use the following notations: Lebesgue measure on $\mathbb{R}^2$ is denoted $\textit{m}$, and is to be understood as the measure defining any $L^p$ space and $W^{1,p}$ space, except when stated otherwise; $\nu$ denotes the measure on $\Pi$ having density $r$ with respect to $\textit{m}$, $|\cdot|$ denotes $\nu$ measure; $B_\delta(y)$ denotes the open ball in $\mathbb{R}^2$ of radius $\delta$ centered at $y$; $\chi_{_\Omega}$ denotes the characteristic function of $\Omega \subseteq \mathbb{R}^2$.

Let $f:\mathbb{R}\to \mathbb{R}$ be a real function. For technical reasons, we make the following assumptions on $f$:
\begin{itemize}\label{3-1}
\item[($\text{f}_1)$] $f$ is locally H\"older continuous on $\mathbb{R}\backslash\{0\}$, $f(s)\equiv 0$ for $s\le 0$, and $f$ is strictly increasing in $(0,+\infty)$ with $\lim_{s \to 0^+}f(s)=\gamma \ge 0$;
\item[$(\text{f}_2)$] there exists some positive number $\delta_0\in(0,1)$ such that
\[\int_{0}^{s}\left(f(t)-\gamma\right)dt \le \delta_0 \left(f(s)-\gamma\right)s, \ \ \forall~ s\ge 0. \]
\item[$(\text{f}_3)$] For all $\tau >0$,
$$\lim_{s\to +\infty}\left(f(s)\,e^{-\tau s}\right)=0.$$
\end{itemize}
 Many nonlinearities that frequently appear in nonlinear elliptic equations satisfy ($\text{f}_1)$--($\text{f}_3)$, for instance $f(s)=s_+^p$ with $p\in(0,+\infty)$. Note that $f$ may have a simple discontinuity, which corresponding to a jump in vorticity
at the boundary of the cross-section of the vortex ring.

Our main result in this papar is as follows.
\begin{theorem}\label{thm}
Suppose $f$ satisfies $(\text{f}_1)$-$(\text{f}_3)$. Then for every $W>0$, $\kappa>0$ and all sufficiently large $\lambda>0$, there exists a classical solution $\Psi^\lambda \in C^1(\Pi)\cap C^2(\Pi\backslash \partial A^\lambda)$ of $(\mathscr{P})$ with
\begin{itemize}
  \item [(i)] $A^\lambda=\{(r,z)\in\Pi \mid \Psi^\lambda(r,z)>0\},\ \ \mu^\lambda= \frac{3\kappa^2}{64\pi^2W}\log \lambda+O(1),\ \ \mathscr{W}^\lambda= \frac{W}{2}\log\lambda.$
  \item [(ii)]Let $\omega^\lambda=\lambda r f(\Psi^\lambda)$ be the azimuthal vorticity, then $\int_\Pi \omega^\lambda d\textit{m} \equiv \kappa$ and $\text{supp}(\omega^\lambda)=A^\lambda$.
  \item [(iii)] There exists some $R_0>1$ not depending on $\lambda$ such that $diam A^\lambda \le R_0 \lambda^{-{1}/{2}}$. Moreover, one has
  \begin{equation*}
  \lim_{\lambda \to +\infty }dist\left(A^\lambda, \left(\frac{\kappa}{4\pi W},0\right) \right)=0.
\end{equation*}
\end{itemize}
\end{theorem}

\begin{remark}
  Kelvin and Hicks showed that if the vortex ring with circulation $\kappa$ has radius $r_*$ and its cross-section $\varepsilon$ is small, then the vortex ring moves at the velocity\,(see, e.g., \cite{Lamb})
\begin{equation*}
  \frac{\kappa}{4\pi r_*}\Big(\log \frac{8r_*}{\varepsilon}-\frac{1}{4}\Big).
\end{equation*}
One can see that our result is consistent with this Kelvin-Hicks formula.
\end{remark}

Our strategy for the proofs of Theorem \ref{thm} is as follows. We transform the problem into a variational problem for the potential vorticity $\zeta:=\omega^\theta/r$, the solutions of which define the desired steady vortex rings. This method is called the vorticity method, which is actually a dual variational principle\,(see, e.g., \cite{Be, CWZ}). Another method to study this problem is called the stream-function method, namely, finding a solution of $(\mathscr{P})$ directly\,(see, e.g., \cite{AS, DV, BF1, Ni, Yang2}). In the study of existence, the key step is to select the appropriate variational admissible class. To this end, we introduce an additional parameter $\Lambda$. By choosing $\Lambda$ appropriately, we can prove that each maximizer of the energy functional over the admissible class will yield a desired solution.  In the study of asymptotic behavior, our method is motivated by the work of Turkington \cite{Tur83}\,(see also \cite{FT}). The key idea is to estimate the order of energy as optimally as possible. We will show that in order to
maximize the energy, these solutions must be concentrated.

\section{Proofs}
As in \cite{BB},  we define the inverse of $\mathcal{L}$ as follows.
\begin{definition}
The Hilbert space $H(\Pi)$ is the completion of $C_0^\infty(\Pi)$ with the scalar products
\begin{equation*}
  \langle u,v\rangle_H=\int_\Pi\frac{1}{r^2}\nabla u\cdot\nabla v d\nu.
\end{equation*}
We define inverse $\mathcal{K}$ for $\mathcal{L}$ in the weak solution sense,
\begin{equation}\label{2-1}
  \langle \mathcal{K}u,v\rangle_H=\int_\Pi uv d\nu \ \ \ for\  all\  v\in H(\Pi), \ \ when \ u \in  L^{10/7}(\Pi,r^3drdz).
\end{equation}
\end{definition}
 Notice that assumption $(\text{f}_2)$ implies $\lim_{s\to +\infty}f(s)=+\infty$ (see \cite{Ni}). One can easily check that $(\text{f}_2)$ is in fact equivalent to
\begin{itemize}
 \item[ $(\text{f}_2)'$] There exists $\delta_1\in(0,1)$ such that
\[G(s) \ge \delta_1 s\, g(s),\,\,\forall\,s\geq0,\]
where $g(s)=0$ in $(-\infty,\gamma]$, $g(s)=f^{-1}(s)$ in $(\gamma,+\infty)$, and $G(s)=\int_0^s g(t)dt$.
\end{itemize}
We note that $g\in C(\mathbb{R})$ is a nonnegative increasing function and $G\in C^2(\mathbb{R})$ is a convex function.
\subsection{Variational problem}
Let $\kappa>0$, $W>0$ be fixed and $0<\varepsilon<1$ be a parameter. Let $d=\kappa/(4\pi W)+1$ and $D=\{(r,z)\in \Pi\mid r<d\}$. Define
\begin{equation*}
\mathcal{A}_{\varepsilon,\Lambda}:=\{\zeta\in L^\infty(\Pi)~|~\int_{\Pi}\zeta d\nu \le \kappa, ~ 0\le \zeta \le \frac{\Lambda}{\varepsilon^2},\ \text{and} \ \zeta(r,z)=0 \ \mbox{if}\ r \ge d \},
\end{equation*}
where $\Lambda>\gamma+1$ is a parameter, which will be determined later.
Consider the maximization problem of the following functional over $\mathcal{A}_{\varepsilon,\Lambda}$
$$E_\varepsilon(\zeta)=\frac{1}{2}\int_D{\zeta \mathcal{K}\zeta}d\nu-\frac{W}{2}\log\frac{1}{\varepsilon}\int_{D}r^2\zeta d\nu-\frac{1}{\varepsilon^2}\int_D G(\varepsilon^2\zeta)d\nu.$$
For any $\zeta\in \mathcal{A}_{\varepsilon,\Lambda}$, by classical elliptic estimate, we have $\mathcal{K}\zeta\in W^{2,p}_{\text{loc}}(\Pi)$ for any $1<p<+\infty$. Let $\zeta^*$ denote the Steiner symmetrization of $\zeta$ with respect to the line $z=0$ in $D$ (see Appendix I of \cite{BF1}).

\begin{lemma}\label{le1}
There exists $\zeta=\zeta^{\varepsilon,\Lambda} \in \mathcal{A}_{\varepsilon,\Lambda} $ such that
\begin{equation*}
 E_\varepsilon(\zeta)= \max_{\tilde{\zeta} \in \mathcal{A}_{\varepsilon,\Lambda}}E_\varepsilon(\tilde{\zeta})<+\infty.
\end{equation*}
Moreover,
\begin{equation}\label{6-1}
\zeta^{\varepsilon,\Lambda}=(\zeta^{\varepsilon,\Lambda})^*=\frac{1}{\varepsilon^2}f(\psi^{\varepsilon,\Lambda})\chi_{_{\{0< \psi^\varepsilon <g(\Lambda)\}\cap D}}+\frac{\Lambda}{\varepsilon^2}\chi_{_{\{\psi^{\varepsilon,\Lambda} \ge g(\Lambda)\}\cap D}} \ \ a.e.\  \text{in}\  \Pi,
\end{equation}
where
\begin{equation}\label{6-2}
  \psi^{\varepsilon,\Lambda}=\mathcal{K}\zeta^{\varepsilon,\Lambda}-\frac{Wr^2}{2}\log\frac{1}{\varepsilon} -\mu^{\varepsilon,\Lambda},
\end{equation}
and the Lagrange multiplier $\mu^{\varepsilon,\Lambda} \ge -g(\Lambda)$ is determined by $\zeta^{\varepsilon,\Lambda}$. If $\zeta^{\varepsilon,\Lambda} \not\equiv 0$ and $\mu^{\varepsilon,\Lambda}>0$, then $\int_\Pi \zeta^{\varepsilon,\Lambda} d\nu=\kappa$.
\end{lemma}

\begin{proof}
We may take a sequence $\{\zeta_{j}\}\subset \mathcal{A}_{\varepsilon,\Lambda}$ such that as $j\to +\infty$
\begin{equation*}
  \begin{split}
        E_\varepsilon(\zeta_{j}) & \to \sup\{E_\varepsilon(\tilde{\zeta})~|~\tilde{\zeta}\in \mathcal{A}_{\varepsilon,\Lambda}\}, \\
        \zeta_{j} & \to \zeta\in L^{10/7}({\Pi,\nu})~~\text{weakly}.
  \end{split}
\end{equation*}
It is easily checked that $\zeta\in \mathcal{A}_{\varepsilon,\Lambda}$. Using the standard arguments (see, e.g., \cite{BB}), we may assume that $\zeta_{j}=(\zeta_{j})^*$, and hence $\zeta=\zeta^*$. By Lemma 2.12 of \cite{BB}, we first have
\begin{equation*}
      \lim_{j\to +\infty}\int_\Pi{\zeta_{j} \mathcal{K}\zeta_{j}}d\nu = \int_\Pi{\zeta \mathcal{K}\zeta}d\nu.
\end{equation*}
On the other hand, we have the lower semicontinuity of the rest of terms, namely,
\begin{equation*}
  \begin{split}
    \liminf_{j\to +\infty} \int_{\Pi}r^2\zeta_j d\nu  & \ge  \int_{\Pi}r^2\zeta d\nu, \\
    \liminf_{j\to +\infty}\int_\Pi G(\varepsilon^2 \zeta_j) d\nu   & \ge \int_\Pi G(\varepsilon^2 \zeta)d\nu.
  \end{split}
\end{equation*}
Consequently, we may conclude that $E_\varepsilon(\zeta)=\lim_{j\to +\infty}E_\varepsilon(\zeta_j)=\sup E_\varepsilon$, with $\zeta\in \mathcal{A}_{\varepsilon,\Lambda}$.

We now turn to prove $\eqref{6-1}$. Consider the family of variations of $\zeta$
\begin{equation*}
  \zeta_{(s)}=\zeta+s(\tilde{\zeta}-\zeta),\ \ \ s\in[0,1],
\end{equation*}
defined for arbitrary $\tilde{\zeta}\in \mathcal{A}_{\varepsilon,\Lambda}$. Since $\zeta$ is a maximizer, we have
\begin{equation*}
  \begin{split}
     0 & \ge \frac{d}{ds}E_\varepsilon(\zeta_{(s)})|_{s=0^+} \\
       & =\int_{\Pi}(\tilde{\zeta}-\zeta)\left(\mathcal{K}\zeta-\frac{Wr^2}{2} \log{\frac{1}{\varepsilon}}-g(\varepsilon^2 \zeta)\right)d\nu.
  \end{split}
\end{equation*}
This implies that for any  $\tilde{\zeta}\in \mathcal{A}_{\varepsilon,\Lambda}$, there holds
\begin{equation*}
  \int_{\Pi}\zeta \left(\mathcal{K}\zeta-\frac{Wr^2}{2} \log{\frac{1}{\varepsilon}}-g(\varepsilon^2 \zeta)\right)d\nu \ge \int_{\Pi}\tilde{\zeta}  \left(\mathcal{K}\zeta-\frac{Wr^2}{2} \log{\frac{1}{\varepsilon}}-g(\varepsilon^2 \zeta)\right)d\nu.
\end{equation*}
By an adaptation of the bathtub principle \cite{Lieb}, we obtain
\begin{equation}\label{6-3}
  \begin{split}
    \mathcal{K}\zeta-\frac{Wr^2}{2} \log{\frac{1}{\varepsilon}}-\mu^{\varepsilon,\Lambda}&\ge g(\varepsilon^2 \zeta) \ \ \  whenever\  \zeta=\frac{\Lambda}{\varepsilon^2}, \\
    \mathcal{K}\zeta-\frac{Wr^2}{2} \log{\frac{1}{\varepsilon}}-\mu^{\varepsilon,\Lambda} &=g(\varepsilon^2 \zeta) \ \ \  whenever\  0<\zeta<\frac{\Lambda}{\varepsilon^2}, \\
    \mathcal{K}\zeta-\frac{Wr^2}{2} \log{\frac{1}{\varepsilon}}-\mu^{\varepsilon,\Lambda} & \le g(\varepsilon^2 \zeta) \ \ \  whenever\  \zeta=0,
  \end{split}
\end{equation}
where
\begin{equation}\label{mu}
  \mu^{\varepsilon,\Lambda}=\inf\{t:|\{\mathcal{K}\zeta-\frac{Wr^2}{2} \log{\frac{1}{\varepsilon}}-g(\varepsilon^2 \zeta)>t\}|\le \frac{\kappa \varepsilon^2}{\Lambda} \}\in \mathbb{R}.
\end{equation}
It follows from \eqref{6-3} that
\begin{equation}\label{6-4}
  \mathcal{K}\zeta-\frac{Wr^2}{2} \log{\frac{1}{\varepsilon}}-\mu^{\varepsilon,\Lambda}=0\ \ \text{a.e.}\ \ \{0<\zeta\le \frac{\gamma}{\varepsilon^2}\}.
\end{equation}
Notice that
\begin{equation}\label{6-5}
  \zeta=\mathcal{L}\left(\mathcal{K}\zeta-\frac{Wr^2}{2} \log{\frac{1}{\varepsilon}}-\mu^{\varepsilon,\Lambda}\right).
\end{equation}
Combining \eqref{6-4} and \eqref{6-5}, we conclude that $\textit{m}(\{0<\zeta\le {\gamma}{\varepsilon^{-2}}\})=0$. Furthermore, by Lemma 2.8 of \cite{BB}, $\mathcal{K}\zeta$  uniformly goes to zero with respect to $z$ as $r\to 0^+$. Hence, by virtue of \eqref{mu}, we must have   $\mu^{\varepsilon,\Lambda}\ge -g(\Lambda)$. Now the stated form $\eqref{6-1}$ follows immediately.

When $\zeta\not\equiv0$ and $\mu^{\varepsilon,\Lambda}>0$, we prove $\int_D\zeta d\nu=\kappa$. Suppose not, then we may consider the family of variations of $\zeta$
\begin{equation*}
  \zeta_{(s)}=\zeta+s\phi,\ \ \ s>0,
\end{equation*}
defined for arbitrary $\phi \in L^1\cap L^{\infty}(D)$ satisfying $\phi\ge 0$ a.e. on $\{\zeta<\delta\}$ and $\phi\le 0$ a.e. on $\{\zeta>\Lambda/\varepsilon^2-\delta\}$ for some $\delta>0$. Clearly, $\zeta_{(s)}\in \mathcal{A}_{\varepsilon,\Lambda}$ for all sufficiently small $s>0$.
We then have
\begin{equation*}
  \begin{split}
     0 & \ge \frac{d}{ds}E(\zeta_{(s)})|_{s=0^+} \\
       & =\int_{\Pi}\phi\left(\mathcal{K}\zeta-\frac{Wr^2}{2} \log{\frac{1}{\beta}}-g(\varepsilon^2\zeta)\right)d\nu,
  \end{split}
\end{equation*}
which implies that
\begin{equation}\label{6-6}
  \begin{split}
    \mathcal{K}\zeta-\frac{Wr^2}{2} \log{\frac{1}{\varepsilon}}&\ge g(\varepsilon^2 \zeta) \ \ \  whenever\  \zeta=\frac{\Lambda}{\varepsilon^2}, \\
    \mathcal{K}\zeta-\frac{Wr^2}{2} \log{\frac{1}{\varepsilon}} &=g(\varepsilon^2 \zeta) \ \ \  whenever\  0<\zeta<\frac{\Lambda}{\varepsilon^2}, \\
    \mathcal{K}\zeta-\frac{Wr^2}{2} \log{\frac{1}{\varepsilon}} & \le g(\varepsilon^2 \zeta) \ \ \  whenever\  \zeta=0,
  \end{split}
\end{equation}
Combining $\eqref{6-3}$ and $\eqref{6-6}$, we deduce that
\begin{equation*}
  \{\zeta>0\}=\{\mathcal{K}\zeta-\frac{Wr^2}{2} \log{\frac{1}{\varepsilon}}>0\}=\{ \mathcal{K}\zeta-\frac{Wr^2}{2} \log{\frac{1}{\varepsilon}}>\mu^{\varepsilon, \Lambda}\},
\end{equation*}
which is clearly a contradiction. The proof is thus completed.
\end{proof}

\subsection{Limiting behavior and proof of Theorem \ref{thm}}
Let $K(r,z,r',z')$ be the Green's function of $\mathcal{L}$ in $\Pi$, with respect to zero Dirichlet data and measure $rdrdz$. One has
\begin{equation*}
   K(r,z,r',z')=\frac{rr'}{4\pi}\int_{-\pi}^{\pi}\frac{\cos\theta'd\theta'}{[(z-z')^2+r^2+r'^2-2rr'\cos\theta']^\frac{1}{2}}.
\end{equation*}
One can easily show that\,(see \cite{BB})
\begin{equation*}
  \mathcal{K}\zeta(r,z)=\int_DK(r,z,r',z')\zeta(r',z')r'dr'dz',\ \ \ \forall\,\zeta\in \mathcal{A}_{\varepsilon,\Lambda}.
\end{equation*}
Let us introduce
 \begin{equation}\label{6-7}
 \sigma=[(r-r')^2+(z-z')^2]^\frac{1}{2}/(4rr')^\frac{1}{2}.
\end{equation}
We have the following estimates about $K$, see \cite{CWZ2}.

\begin{lemma}\label{le2}
 There holds
 \begin{equation}\label{6-8}
   0<K(r,z,r',z')\leq\frac{(rr')^\frac{1}{2}}{4\pi}\sinh^{-1}(\frac{1}{\sigma}), \ \ \forall\,\sigma >0.
 \end{equation}
 Moreover, there exists a continuous function $h \in L^{\infty}(\Pi\times \Pi)$ such that
 \begin{equation}\label{6-9}
   K(r,z,r',z')=\frac{\sqrt{rr'}}{2\pi}\log{\frac{1}{\sigma}}+\frac{\sqrt{rr'}}{2\pi}\log(1+\sqrt{\sigma^2+1})+h(r,z,r',z')\sqrt{rr'},\ \ \text{in}\ \Pi\times \Pi.
 \end{equation}
\end{lemma}

The next lemma give a lower bound of energy.
\begin{lemma}\label{le3}
  For any $a\in[b,c]\subset (0,d)$, there exists $C>0$ such that for all $\varepsilon$ sufficiently small, we have
\begin{equation*}
  E_\varepsilon(\zeta^{\varepsilon,\Lambda})\ge \left(\frac{a\kappa^2}{4\pi}-\frac{\kappa W a^2}{2}\right)\log\frac{1}{\varepsilon}-C,
\end{equation*}
where the positive number $C$ depends only on $b$, $c$, but not on $\varepsilon$, $\Lambda$.
\end{lemma}

\begin{proof}
Choose a test function $\tilde{\zeta}^{\varepsilon,\Lambda} \in \mathcal{A}_{\beta,\Lambda}$ defined by $\tilde{\zeta}^{\varepsilon,\Lambda}=\frac{1}{\varepsilon^2} \chi_{_{B_{{\varepsilon}{\sqrt{\kappa /(a\pi)}}}\left((a,0)\right)}}$. Since $\zeta^{\varepsilon,\Lambda}$ is a maximizer, we have $E_\varepsilon(\zeta^{\varepsilon,\Lambda})\ge E_\varepsilon(\tilde{\zeta}^{\varepsilon,\Lambda})$. By Lemma $\ref{le2}$, we obtain
 \begin{equation*}
\begin{split}
    E_\varepsilon(\tilde{\zeta}^{\varepsilon,\Lambda})&= \frac{1}{2}\int_D\int_D\tilde{\zeta}^{\varepsilon,\Lambda}(r,z)K(r,z,r',z')\tilde{\zeta}^{\varepsilon,\Lambda}(r',z')r'rdr'dz'drdz-{\frac{W}{2} \log{\frac{1}{\varepsilon}}}\int_{D}r^2\tilde{\zeta}^{\varepsilon,\Lambda}d\nu\\
                    &\ \ \ \ \ -\frac{1}{\varepsilon^2}\int_D G(\varepsilon^2\tilde{\zeta}^{\varepsilon,\Lambda})d\nu\\
                    &\ge \frac{a+O(\beta)}{4\pi}\int_D\int_D\log{\frac{1}{[(r-r')^2+(z-z')^2]^{1/2}}}\tilde{\zeta}^{\varepsilon,\Lambda}(r,z)\tilde{\zeta}^{\varepsilon,\Lambda}(r',z')r'rdr'dz'drdz\\
                    &\ \ \ \ \ -\frac{\kappa W[a+O(\varepsilon)]^2}{2}\log{\frac{1}{\varepsilon}}-C_1\\
                    &\ge \left(\frac{a\kappa^2}{4\pi}-\frac{\kappa W a^2}{2}\right)\log{\frac{1}{\varepsilon}}-C_2.
\end{split}
\end{equation*}
We therefore complete the proof.
\end{proof}

We now turn to estimate the Lagrange multiplier $\mu^{\varepsilon,\Lambda}$.
\begin{lemma}\label{le4}
Let $\delta_1$ be as in  $(\mbox{f}_2)'$, we have
\[\mu^{\varepsilon,\Lambda}\ge 2\kappa^{-1}E_\varepsilon(\zeta^{\varepsilon,\Lambda})- |1-2\delta_1|g(\Lambda)+ \frac{\kappa^{-1}W}{2}\log\frac{1}{\varepsilon}\int_\Pi r^2 \zeta^{\varepsilon,\Lambda} d\nu-C,\]
\end{lemma}
where the positive number $C$ does not depend on $\varepsilon$, $\Lambda$.
\begin{proof}
  Recalling \eqref{6-3} and the assumption $(\text{f}_2)'$, we have
  \begin{equation}\label{6-10}
    \begin{split}
       2E_\varepsilon(\zeta^{\varepsilon,\Lambda}) & =\int_\Pi \zeta^{\varepsilon,\Lambda} \mathcal{K}\zeta^{\varepsilon,\Lambda} d\nu -W\log\frac{1}{\varepsilon}\int_\Pi r^2 \zeta^{\varepsilon,\Lambda} d\nu -\frac{2}{\varepsilon^2}\int_\Pi G(\varepsilon^2\zeta^{\varepsilon,\Lambda})d\nu \\
         & \le \int_\Pi \zeta^{\varepsilon,\Lambda} \psi^{\varepsilon,\Lambda} d\nu -\frac{W}{2}\log \frac{1}{\varepsilon}\int_\Pi r^2 \zeta^{\varepsilon,\Lambda} d\nu -2\delta_1 \int_\Pi \zeta^{\varepsilon,\Lambda}g(\varepsilon^2\zeta^{\varepsilon,\Lambda})d\nu +\kappa \mu^{\varepsilon,\Lambda} \\
         & =(1-2\delta_1)\int_\Pi \zeta^{\varepsilon,\Lambda}g(\varepsilon^2\zeta^{\varepsilon,\Lambda})d\nu+\int_\Pi \zeta^{\varepsilon,\Lambda}[\psi^{\varepsilon,\Lambda}-g(\Lambda)]_+d\nu\\
         &\ \ \ \ \ \ \ \ -\frac{W}{2}\log \frac{1}{\varepsilon}\int_\Pi r^2 \zeta^{\varepsilon,\Lambda} d\nu+\kappa\mu^{\varepsilon,\Lambda} \\
         & \le |1-2\delta_1|\kappa g(\Lambda) +\int_\Pi \zeta^{\varepsilon,\Lambda}[\psi^{\varepsilon,\Lambda}-g(\Lambda)]_+d\nu-\frac{W}{2}\log \frac{1}{\varepsilon}\int_\Pi r^2 \zeta^{\varepsilon,\Lambda} d\nu+\kappa\mu^{\varepsilon,\Lambda}.
    \end{split}
  \end{equation}
   Set $U^{\varepsilon,\Lambda}:=[\psi^{\varepsilon,\Lambda}-g(\Lambda)]_+$. By integration by parts, one has
\begin{equation}\label{6-11}
 \int_\Pi {\frac{|\nabla U^{\varepsilon,\Lambda}|^2}{r^2}}d\nu= \int_\Pi \zeta^{\varepsilon,\Lambda}U^{\varepsilon,\Lambda}d\nu.
\end{equation}
On the other hand, by H\"older's inequality and Sobolev embedding
\begin{equation}\label{6-12}
\begin{split}
  \int_\Pi \zeta^{\varepsilon,\Lambda}U^{\varepsilon,\Lambda}d\nu
     & \le \frac{\Lambda}{\varepsilon^2}|\{\zeta^{\varepsilon,\Lambda}={\Lambda}{\varepsilon^{-2}}\}|^{\frac{1}{2}}\left(\int_D |U^{\varepsilon,\Lambda}|^2d\nu\right)^{\frac{1}{2}}\\
     &\le \frac{d\Lambda}{\varepsilon^2}|\{\zeta^{\varepsilon,\Lambda}={\Lambda}{\varepsilon^{-2}}\}|^{\frac{1}{2}}\left(\int_D |U^{\varepsilon,\Lambda}|^2d\textit{m}\right)^{\frac{1}{2}}\\
     & \le \frac{Cd\Lambda }{\varepsilon^2}|\{\zeta^{\varepsilon,\Lambda}={\Lambda}{\varepsilon^{-2}}\}|^{\frac{1}{2}}\int_D |\nabla U^{\varepsilon,\Lambda}|d\textit{m} \\
     & \le Cd\left(\int_\Pi{\frac{|\nabla U^{\varepsilon,\Lambda}|^2}{r^2}}d\nu\right)^{\frac{1}{2}}.
\end{split}
\end{equation}
Here the positive constant $C$  does not depend on $\varepsilon$ and $\Lambda$. Combining \eqref{6-11} and \eqref{6-12}, we conclude that $\int_\Pi \zeta^{\varepsilon,\Lambda}U^{\varepsilon,\Lambda}d\nu $ is uniformly bounded with respect to $\varepsilon$ and $\Lambda$. Now the desired result clearly follows from \eqref{6-10}.
\end{proof}

\begin{corollary}\label{le5}
If $\varepsilon$ is small enough, then $\int_\Pi \zeta^{\varepsilon,\Lambda}d\nu=\kappa$.
\end{corollary}

We can repeat the arguments in \cite{CWZ, CWZ2} to obtain the following asymptotics.
\begin{lemma}\label{le6}
There holds
\begin{equation*}
diam[\text{supp}(\zeta^{\varepsilon,\Lambda})] \le 4d \varepsilon^{{1}/{2}}
\end{equation*}
provided $\varepsilon$ is small enough. Moreover, we have
\begin{equation*}
  \lim_{\varepsilon \to 0^+}dist\left( \text{supp}(\zeta^{\varepsilon,\Lambda}), (\frac{\kappa}{4\pi W},0) \right)=0.
\end{equation*}
\end{lemma}

Recalling \eqref{6-1}, $\zeta^{\varepsilon,\Lambda}$ has an extra patch part. We now eliminate that patch part by selecting sufficiently large $\Lambda$. To this end, we first show that $\psi^{\varepsilon,\Lambda}$ has a prior upper bound with respect to $\Lambda$.
\begin{lemma}\label{le7}
We have
\begin{equation*}
  \psi^{\varepsilon,\Lambda}(r,z)\le |1-2\delta_1|g(\Lambda)+ C \log \Lambda +C, \ \ \forall~ (r,z)\in D,
\end{equation*}
where the positive constant $C$ is independent of $\varepsilon$, $\Lambda$.
\end{lemma}
\begin{proof}
In view of Lemma \ref{le2}, one has
\begin{equation}\label{6-13}
  0<K(r,z,r',z')\le \frac{\sqrt{rr'}}{2\pi}\log\frac{1}{|(r-r')^2+(z-z')^2|^\frac{1}{2}}+C, \ \ \text{in}\ \ D\times D.
\end{equation}
Here $C>0$ depend on $\varepsilon$ and $\Lambda$.
For any $x=(r,z)\in \text{supp}(\zeta^{\varepsilon,\Lambda})$, we have
\begin{equation*}
\begin{split}
   \mathcal{K}\zeta^{\varepsilon,\Lambda}(x) &  \le \frac{r+C\varepsilon^\frac{1}{2}}{2\pi}\int_D\log\frac{1}{|x-y|}\zeta^{\varepsilon,\Lambda}(y)d\nu(y)+C \\
     &\le \frac{r^2}{2\pi} \int_D\log\frac{1}{|x-y|}\zeta^{\varepsilon,\Lambda}d\textit{m}(y)+C\varepsilon^\frac{1}{2} \log \Lambda+C\\
     &\le \frac{r^2}{2\pi}\left(\log\frac{1}{\varepsilon}+\frac{\log \Lambda}{2}\right)\int_D\zeta^{\varepsilon,\Lambda}d\textit{m}(y)+C\varepsilon^\frac{1}{2} \log \Lambda+C\\
     & \le \frac{\kappa r}{2\pi}\log\frac{1}{\varepsilon}+C\log\Lambda+C \\
\end{split}
\end{equation*}
Hence
\begin{equation}\label{6-14}
  \psi^{\varepsilon,\Lambda}(r,z)\le  \left(\frac{\kappa r}{2\pi}-\frac{Wr^2}{2}\right)\log\frac{1}{\varepsilon}-\mu^{\varepsilon,\Lambda}+C\log\Lambda+C
\end{equation}
On the other hand, by Lemmas \ref{le3} and \ref{le4}, we have
\begin{equation}\label{6-15}
\begin{split}
   \mu^{\varepsilon,\Lambda} & \ge 2\kappa^{-1}E_\varepsilon(\zeta^{\varepsilon,\Lambda})- |1-2\delta_1|g(\Lambda)+ \frac{\kappa^{-1}W}{2}\log\frac{1}{\varepsilon}\int_\Pi r^2 \zeta^{\varepsilon,\Lambda} d\nu-C \\
     & \ge \left(\frac{\kappa r}{2\pi}-\frac{Wr^2}{2}\right)\log\frac{1}{\varepsilon}- |1-2\delta_1|g(\Lambda)-C.
\end{split}
\end{equation}
Combining \eqref{6-14} and \eqref{6-15}, we conclude that

\begin{equation*}
  \psi^{\varepsilon,\Lambda}(r,z) \le |1-2\delta_1|g(\Lambda)+ C \log \Lambda +C.
\end{equation*}
The proof is completed.
\end{proof}

Using Lemma \ref{le7}, we can now eliminate that patch part.
\begin{lemma}\label{le8}
  There exists some $\Lambda_0>1$ such that
\begin{equation*}
  \zeta^{\varepsilon,\Lambda_0}=\frac{1}{\varepsilon^2}f\left(\psi^{\varepsilon,\Lambda_0}\right)\chi_{_D}
\end{equation*}
provided $\varepsilon$ is small enough.
\end{lemma}

\begin{proof}
  Notice that
\begin{equation}\label{6-16}
  \psi^{\varepsilon,\Lambda} \ge g(\Lambda),\ \ \text{a.e.\ on}\ \  \{\zeta^{\varepsilon,\Lambda}=\frac{\Lambda}{\varepsilon^2}\}.
\end{equation}
In view of Lemma \ref{le7}, there is some $C>0$ not depending on $\varepsilon$ and $\Lambda$ satisfying
\begin{equation}\label{6-17}
   \psi^{\varepsilon,\Lambda}(r,z)\le |1-2\delta_1|g(\Lambda)+ C \log \Lambda +C, \ \ \forall~ (r,z)\in D.
\end{equation}
Combining \eqref{6-16} and \eqref{6-17}, we get
\begin{equation*}
   g(\Lambda)\le |1-2\delta_1|g(\Lambda)+ C \log \Lambda +C, \ \ \text{a.e.\ on}\ \  \{\zeta^{\varepsilon,\Lambda}=\frac{\Lambda}{\varepsilon^2}\},
\end{equation*}
that is,
\begin{equation*}
  (1-|1-2\delta_1|)g(\Lambda)\le C \log \Lambda +C, \ \text{a.e.\ on}\ \  \{\zeta^{\varepsilon,\Lambda}=\frac{\Lambda}{\varepsilon^2}\}.
\end{equation*}
Recalling the assumption $(\text{f}_3)$, we have
\begin{equation*}
  \lim_{\Lambda \to +\infty}\left((1-|1-2\delta_1|)g(\Lambda)-C\log\Lambda \right)=+\infty.
\end{equation*}
Therefore, we must have $\textit{m}(\{\zeta^{\varepsilon,\Lambda}={\Lambda}{\varepsilon^{-2}}\})=0$ if $\Lambda$ is sufficiently large, say $\Lambda \ge \Lambda_0$. The proof is completed.
\end{proof}

For the rest of this section, we shall abbreviate $(\mathcal{A}_{\varepsilon, \Lambda_0}, \zeta^{\varepsilon, \Lambda_0},\psi^{\varepsilon, \Lambda_0}, \mu^{\varepsilon, \Lambda_0})$ as $(\mathcal{A}_\varepsilon, \zeta^\varepsilon, \psi^{\varepsilon}, \mu^{\varepsilon})$. We now further study the profile of $\zeta^\varepsilon$.
\begin{lemma}\label{le9}
  For all sufficiently small $\varepsilon$, there holds
\begin{equation}\label{6-18}
  \zeta^\varepsilon=\frac{1}{\varepsilon^2}f\left(\psi^\varepsilon\right).
\end{equation}
\end{lemma}

\begin{proof}
In view of Lemma \ref{le6}, we have
\begin{equation*}
  dist\left(\text{supp}(\zeta^\varepsilon), \partial D\right)>0
\end{equation*}
provided $\varepsilon$ is sufficiently small. Hence
\begin{equation*}
  \begin{split}
      \mathcal{L}\psi^\varepsilon&=0\ \ \text{in}\  \Pi \backslash  supp(\zeta^\varepsilon), \\
      \psi^\varepsilon&\le 0\ \ \text{on}\  \partial \Pi \cup \partial \big(supp(\zeta^\varepsilon)\big), \\
     \psi^\varepsilon &\le 0 \ \ \text{at}\  \infty .
  \end{split}
\end{equation*}
By the maximum principle, we conclude that $\psi^\varepsilon\le 0$ in $\Pi \backslash D$.
The proof is completed.
\end{proof}

We can sharpen Lemma \ref{le6} as follows. For proof, we refer to \cite{CWZ}\,(see also \cite{CWZ2}).
\begin{lemma}\label{sharp}
  There exists a $R_0>1$ not depending on $\varepsilon$ such that
\begin{equation*}
diam[\text{supp}(\zeta^{\varepsilon,\Lambda})] \le R_0 \varepsilon,
\end{equation*}
provided $\varepsilon$ is small enough.
\end{lemma}

Combining Lemmas \ref{le3}, \ref{le4} and \ref{le6}, we get the following asymptotic expansions.
\begin{lemma}\label{le10}
As $\varepsilon \to 0^+$, we have
\begin{align}
\label{219} E_\varepsilon(\zeta^\varepsilon) & =\left(\frac{\kappa^2r_*}{4\pi}-\frac{\kappa Wr_*^2}{2}\right)\log{\frac{1}{\varepsilon}}+O(1), \\
\label{220}  \mu^\varepsilon & =\left(\frac{\kappa r_*}{2\pi}-\frac{Wr_*^2}{2}\right)\log{\frac{1}{\varepsilon}}+O(1),
\end{align}
where $r_*=\kappa/(4\pi W)$.
\end{lemma}

Define the center of $\zeta^\varepsilon$ to be
\begin{equation*}
  x^\varepsilon=\frac{\int_\Pi x\zeta^\varepsilon(x)d\textit{m}(x)}{\int_\Pi \zeta^\varepsilon(x)d\textit{m}(x)}.
\end{equation*}
Let $\eta^\varepsilon(x)=\varepsilon^2\zeta^\varepsilon(x^\varepsilon+\varepsilon x)$. The following result determines the asymptotic nature of $\zeta^\varepsilon$ in terms of its scaled version $\eta^\varepsilon$. Since the proof of is very similar to that in \cite{CWZ2} without significant changes, we omit it here.

\begin{lemma}\label{le11}
Every accumulation point of the family $\{\eta^{\varepsilon}:\varepsilon>0\}$ in the weak topology of $L^2$ must be a radially nonincreasing function.
\end{lemma}

\begin{proof}[Proof of Theorem \ref{thm}]
Let $\lambda=1/\varepsilon^2$ and $\Psi^\lambda=\psi^\varepsilon$. The regularity of $\Psi^\lambda$ follows from classical Schauder's theory. And the desired results follows from the above lemmas.
\end{proof}

{\bf Acknowledgments.}
{ This work was supported by NNSF of China (No.11771469) and Chinese Academy of Sciences (
  No.QYZDJ-SSW-SYS021).
}

\phantom{s}
 \thispagestyle{empty}

\end{document}